\documentclass[10pt]{article}
\usepackage{float}
\usepackage{pifont}
\usepackage{geometry}                
\usepackage{amssymb}
\usepackage{amsmath,amsthm}
\usepackage{amsfonts}
\usepackage{fancybox}
\usepackage{empheq}
\usepackage{mathtools}
\usepackage{calrsfs}
\usepackage[parfill]{parskip}    

\usepackage{titletoc}
\titlecontents{section}[2.3em]
  {}
  {\bfseries\contentslabel[\thecontentslabel.]{2em}\footnotesize\MakeUppercase}
  {\hspace*{-2.3em}\MakeUppercase}
  {\titlerule*[1pc]{.}\contentspage}
\titlecontents{subsection}[4.6em]
  {}
  {\contentslabel{2em}\footnotesize}
  {\hspace*{-2.3em}}
  {\titlerule*[1pc]{.}\contentspage}
\titlecontents{subsubsection}[6.9em]
  {}
  {\contentslabel{2em}\footnotesize\itshape\space}
  {\hspace*{-2.3em}}
  {\titlerule*[1pc]{.}\contentspage}

\makeatletter
\renewcommand\tableofcontents{%
  \section*{\centerline{\small\MakeUppercase{\contentsname}}
    \@mkboth
      {\MakeUppercase\contentsname}
      {\MakeUppercase\contentsname}
  }%
  \@starttoc{toc}%
}
\makeatother

\usepackage{hyperref}

\usepackage{enumitem}

\usepackage{framed}
\usepackage{parskip}
\usepackage{color, colortbl}
\definecolor{LightCyan}{rgb}{0.88,1,1}
\definecolor{Gray}{gray}{0.9}
\usepackage{soul}

\usepackage[all,cmtip]{xy}

\usepackage{color}
\usepackage{bm}
\usepackage{comment}
\usepackage{cancel}

\usepackage{todonotes}
\usepackage{cleveref}

\usepackage{tikz}

\newtheorem{theorem}{Theorem}
\newtheorem{lemma}[theorem]{Lemma}
\newtheorem{corollary}[theorem]{Corollary}
\newtheorem{proposition}[theorem]{Proposition}
\newtheorem{example}[theorem]{Example}
\newtheorem{remark}[theorem]{Remark}

\newcommand{\Char}{\operatorname{char}}

\newcommand{\Z}{\mathbb{Z}}

\newcommand{\mf}{\mathfrak{p}}

\newcommand{\Gal}{\operatorname{Gal}}

\date \today

\makeatletter
\newcommand{\aprod}{\mathop{\operator@font \hbox{\Large$\ast$}}}
\makeatother

\begin{document}
\sloppy

\title{On the Hasse - Arf property of local fields}
\author{Ioannis Tsouknidas}
\newcommand{\Addresses}{{
  \bigskip
  \footnotesize
  I.~Tsouknidas, \textsc{Beijing Institute of Applied Mathematical and Physical Sciences (BIMSA)}\par\nopagebreak
  \textit{E-mail address:} 
  \texttt{kharga91@gmail.com (preferred),\, iotsouknidas@bimsa.cn}
}}

\maketitle
\tableofcontents

\begin{abstract} Let $F/K$ be a finite Galois totally \& wildly ramified extension of complete discrete valuation fields. We say that the extension has the Hasse - Arf property if the ramification jumps in upper numbering are integers. We give necessary defining equations for $F$ in terms of the ramification jumps. 
 In order for the Hasse - Arf property to hold, these equations become very strict. We prove that the last assertion is an equivalence condition, thus in terms of these defining equations, the Hasse - Arf property becomes an equivalence condition. 
 
\end{abstract}

\section{Introduction}
\renewcommand{\thefootnote}{\fnsymbol{footnote}} 
\footnotetext{\emph{MSC classification:} 
11S15, 11S20. 
\\\emph{Key words:} Ramification theory, Hasse-Arf property.}    
\renewcommand{\thefootnote}{\arabic{footnote}}


 The Hasse - Arf theorem is one of the foundational theorems of number theory. Originally proved by Hasse (\cite{Hasse_1930}, \cite{zbMATH03012900})  and generalized by Arf (\cite{zbMATH02509898}),
 it continues to accumulate interest to these days (see \cite{Fesenko_1995}, \cite{zbMATH07974418}, \cite{elder2025artinschreierwittextensionsramificationbreaks} and the references therein). Probably the most basic but important result for this work is lemma \ref{thm:mainlemma}, which is related to 
the ``special element'' of Sen \cite[prop. pg. 37]{Sen_1969}. This allows an understanding of ramification through elements which are not uniformizing.  

The author would like to thank Beihui Yuan for assistance with calculations using Macaulay2. Aristides Kontogeorgis, Jie Wu, Sergey Oblezin, Yongxiong Li and Beihui Yuan provided helpful discussions. This work is supported by a fund by the Beijing government.

\subsection{An application of Lucas's theorem}
We assume that $0^0=1$.
 This section is needed for lemma \ref{lem:laurentexpansion} which results in lemma \ref{thm:mainlemma}.

 Let $p$ be a prime number and by $v_p$ we denote the $p-$adic valuation of the rational numbers. Let $m$ be a positive and $n$ a non-negative integer. Write $m=\sum_{i=0}^{\mu}m_ip^i$, $n=\sum_{i=0}^{\nu}n_ip^i$ for their base$-p$ expansions (so that $0\leq m_i,n_i<p$ for all $i$). For a binomial coefficient $\binom{m}{n},$ if $n>m$ we set $\binom{m}{n}=0.$

Lucas's theorem \cite[Theorem. 15.22]{Eisenbud:95} states that:
\begin{equation}\label{thm:lucastheorem}
  {m\choose n}\equiv\prod_{i=1}^{\max\{\mu,\nu\}}{m_i\choose n_i}\mod p.
\end{equation}
That in particularly yields that the binomial coefficient ${m\choose n}$ vanishes modulo $p$ if $n_i>m_i$ for some $i.$

If $m$ is negative then the binomial coefficient is defined as:
\begin{equation}\label{eq:binomial}
  {m\choose n}=\frac{m(m-1)\cdots (m-n+1)}{n!}
\end{equation}
for $n>0$ and for $n=0$ we set ${m\choose 0}=1.$

\begin{lemma}\label{lem:luc}
  Let $m\in\mathbb{Z}\backslash\{0\}$. The least \emph{positive} integer $n$ such that ${m\choose n}\not\equiv 0\mod p$ is $p^{v_p(m)}$.
\end{lemma}

\begin{proof} If $m>0$, writing $m=p^{v_p(m)}\mu$ with $\mu>0$ prime to $p$, and applying Lucas's theorem, gives immediately that for $0<n<p^{v_p(m)}$ the binomial ${m\choose n}$ is zero modulo $p.$ We calculate now for $n=p^{v_p(m)}.$ If $\mu=\sum_{i=0}\mu_ip^i$ with $0\leq\mu_i<p,$ and $\mu_0\neq 0$ (since $\mu$ was assumed prime to $p$) then by Lucas's theorem:
\[
  {m\choose p^{v_p(m)}}\equiv {\mu_0\choose 1}\cdot \prod_{i=1}{\mu_i\choose 0}=\mu_0\not\equiv 0\mod p.
\]

Now let $m<0$ and $v_p(m)=\xi$ so that $m=-\mu p^\xi$ for a \emph{positive} integer $\mu$, prime to $p.$ We need to show that the least positive integer $n$ such that ${m\choose n}$ is nonzero is $p^\xi.$ We write:
  \[
    {m\choose n}=(-1)^n{-m+n-1\choose n}=(-1)^n{\mu p^\xi+n-1\choose n}.
  \]
  Notice that the binomial on the right side of last equality has positive factors and therefore Lucas's theorem can be applied. We first show that for $0<n<p^\xi$ the binomial coefficient ${\mu p^\xi+n-1\choose n}$ is zero.

   Since $0<n<p^\xi,$ writing the $p-$base expansion of $n$ and $n-1;$ $n=\sum_{0}^{<\xi}n_ip^i$ and $n-1=\sum_{0}^{<\xi}n_i'p^i,$ we see that $n_i>n_{i'}$ for at least one $i$. Also $\mu p^\xi$ does not affect that since $\mu$ is positive and then by Lucas's theorem, ${\mu p^\xi+n-1\choose n}$ is indeed zero modulo $p$ for $0<n<p^\xi=p^{v_p(m)}.$

It remains to show that ${m\choose p^\xi}={\mu p^\xi+p^\xi-1\choose p^\xi}$ is non-zero. Let $\mu=\sum_{i=0}^{<\infty}\mu_ip^i$ be the base-$p$ expansion of $\mu.$ Since $p\nmid \mu$, $\mu_0$ is non-zero.
  The $p-$adic expansion of $\mu p^\xi+p^\xi-1$ is:
  \[
    p^\xi-1+\mu p^\xi=(p-1)p^0+\dots+ (p-1)p^{\xi-1}+\mu_0p^\xi+\sum_{i=1}\mu_ip^{i+\xi},
  \]
  while $p^\xi$ is $0+0p+\dots+1p^{\xi}+0.$
  By Lucas's theorem:
  \[
  {m\choose p^\xi}=(-1)^{p^\xi}{\mu p^\xi+p^\xi-1\choose p^\xi}\equiv(-1)^{p^\xi}{\mu_0\choose 1}=(-1)^{p^\xi} \mu_0\not\equiv 0\mod p.
  \]
  Notice that the integer $\mu_0$ in the second case of this proof (for negative $m$) is $\mu_0\equiv-m/p^{\xi}\mod p$ and $\xi=v_p(m).$\end{proof}

\begin{remark}\label{remark:Luc}
  For  
   a negative integer $m$ we have shown that:
  \[
  {m\choose p^{v_p(m)}}\equiv(-1)^{p^{v_p(m)}+1}\frac{m}{p^{v_p(m)}} \mod p. 
  \]
  Since $(-1)^{p^{v_p(m)}+1}$ equals $1$ for all primes $p,$ we have shown that for $m\in\mathbb{Z}\backslash\{0\}:$
  $${m\choose p^{v_p(m)}}\equiv \frac{m}{p^{v_p(m)}} \mod p.$$
\end{remark}

\begin{corollary}\label{cor:lucas}
  Let $p$ denote either a prime number or \emph{zero}. We denote by $v_{p}$ the $p-$adic valuation of $p$ if $p\neq 0$ and define $v_0$ to be the zero function. We make the convention $0^0=1.$ Then for $m\in\mathbb{Z}\backslash\{0\}$ the following holds:
  \[
  {m\choose p^{v_p(m)}}\equiv\frac{m}{p^{v_p(m)}}\mod p.
  \]
  \end{corollary}
  \begin{proof}
    For $p$ a prime number it is remark \ref{remark:Luc}. If $p=0$ it is immediate.
  \end{proof}

\begin{lemma}\label{lem:padicbinomial}
  If $0<i<m$ are positive integers then:
  \[
  v_p\binom{m}{i}\geq v_p(m)-v_p(i)=v_p\left(\frac{m}{i}\right).
  \]
\end{lemma}
\begin{proof}
  We know that:
  \[
  \binom{m}{i}=\frac{m}{i}\binom{m-1}{i-1},
  \]
  and the binomial coefficient $\binom{m-1}{i-1}$ is an integer, therefore it has nonnegative $p-$adic valuation, therefore:
  \[
  v_p\binom{m}{i}=v_p\left(\frac{m}{i}\right)+v_p\binom{m-1}{i-1}\geq v_p\left(\frac{m}{i}\right).
  \]
\end{proof}

\section{Ramified extensions of complete discrete valuation fields}
Let $K$ be a complete discrete valuation field with 
valuation ring (resp. valuation, resp. maximal ideal) $A_K$ (resp. $v_K,$ resp $\mathfrak{p}_K$).
 We assume always that the residue field has positive characteristic $p$ and $v_p$ denotes the $p-$adic valuation of the rational numbers. 
\subsubsection*{Between equal and mixed characteristic}

If $K$ has positive characteristic and $m$ is a nonzero integer we identify $\binom{m}{n}$ with the corresponding element as in Lucas's theorem (\ref{thm:lucastheorem}), under the usual convention that $\binom{m_i}{n_i}=0$ when $n_i>m_i.$ In that way, the binomial coefficient $\binom{m}{n}$ is well defined in a field of arbitrary characteristic for $m\in\Z\backslash\{0\}$ and nonnegative integer $n.$

\begin{lemma}\label{lem:laurentexpansion}
Let $m\in\Z\backslash\{0\}$ and  $a\in \mf_K$. If $\Char K>0$ or $(m,p)=1$ then the element:
  \[
  (1+a)^m=1+\frac{m}{p^{v_p(m)}}a^{p^{v_p(m)}}+\,\aprod
  \]
  lies in $K$ and  ``$\aprod$'' stands for elements with greater valuation.
\end{lemma}
In other words, if $\Char K>0$ or, $\Char K=0$ and $v_p(m)=0,$ then   $(1+a)^m\equiv 1+\frac{m}{p^{v_p(m)}}a^{p^{v_p(m)}}\mod \mathfrak{p}_K^{p^{v_p(m)}+1}.$

\begin{proof}
  Since $K$ is complete and $v_K(a)>0,$ the sequence $\{a^n\}_{n\geq 0}$ tends to zero and the series $\sum_{n=0}^{\infty}{m\choose n}a^n$ lies in $K$. 
  \begin{description}
     
     \item[If char$ \bm{K>0}$ :]  It was shown in lemma \ref{lem:luc} and corollary \ref{cor:lucas}. 
     \item[If char$ \bm{K=0}$ :] We have that $v_p(m)=0$ and we want to show that for $i>1$:
     \[
  v_K(ma)\lneq v_K\left(\binom{m}{i}\right)+v_K(a^i).
  \]
  Write $v_K(ma)= e_0\cdot v_p(m)+v_K(a)=v_K(a).$ 
  On the other hand, the binomial on the right side is always an integer (regardless $m$ being positive or not) so that its $p-$adic valuation is nonnegative and $v_K(a^i)=iv_K(a)>v_K(a)$ since $i>1.$ \end{description}
\end{proof}

\subsection{Totally ramified extensions}
 We consider now a \textbf{finite Galois} extension $F/K$ with Galois group $\Gal(F/K).$ Denote by $A_F$ (resp. $v_F,$ resp. $\mf_F$) the valuation ring (resp. the valuation, resp. the maximal ideal) of $F.$ Write $e$ for the ramification index of $F/K.$
 One can study the ``higher'' ramification groups (those of nonnegative index) by replacing $K$ with the largest unramified subextension of $F$ over $K$ (\cite[IV.§1, 1st corollary]{serre2013local}). Since this is our purpose, we assume that $F/K$ is \textbf{totally ramified:} $[F:K]$ equals the ramification index of $F/K.$ Denote by $\pi_F,$ $\pi_K$ the uniformizers of $A_F$ and $A_K.$ As before, $e$ is the ramification index of $F$ over $K.$ The elements $1,\pi_F,\dots,\pi_F^{e-1}$ form a basis of $A_F$ over $A_K$ (\cite[III.§6, prop. 12]{serre2013local}).

%

\subsubsection*{Ramification groups}
As before, $F/K$ is a finite Galois extension of complete discrete valuation fields and is totally ramified, so that the residue fields $\bar{F},$ $\bar{K}$ are equal and the ramification index equals the degree of the extension.
  The \textbf{ramification groups} of the extension are defined for $i\geq-1$ as:
\[
  G_i:=\{\sigma\in G: v[\sigma(\pi_F)-\pi_F]\geq i+1\}.
\]

For $\sigma\in G$ denote by $i_G(\sigma)$ the integer  $v_F[(\sigma(\pi_F)-\pi_F)/\pi_F].$ We always write $p$ for the positive characteristic of the residue field. The following lemma is related to 
\cite[prop. pg. 37]{Sen_1969}. 
\begin{lemma}\label{thm:mainlemma}
Let $\sigma\in G_1$ 
 and $f\in F\backslash\{0\}.$ If $v_F(f)$ is nonzero and prime to $p$ 
 then: 
 $$v_F[\sigma(f)-f]= v_F(f)+i_G(\sigma).$$
  
\end{lemma}

\begin{proof}
If $\sigma=1$ we have infinity on both sides. 
 For the uniformizer $\pi_F$ we have $v_F(\sigma(\pi_F)-\pi_F)=1+i_G(\sigma)$ and we write:
\[
  \sigma(\pi_F)=
  \pi_{F}\cdot(1+a)\text{ with }a\in\mf^{i_G(\sigma)}.
\]
From now and until the end of this proof we write $i$ for $i_G(\sigma).$ 
 Let $f\in F$ under our assumptions. Then $f=\pi_F^{v_F(f)}\mu,$ $v_F(\mu)=0$ and:
\[
  \sigma(f)=\sigma(\pi_F)^{v_F(f)}\sigma(\mu).
\]
We calculate separately:
\begin{description}
    \item[For $\bm{\sigma(\pi_F)^{v_F(f)}}$:]
     We have $\sigma(\pi_F)^{v_F(f)}=\pi_F^{v_F(f)}(1+a)^{v_F(f)}.$ Recall that $\sigma\in G_1$ so $i=i_G(\sigma)\geq 1,$ hence $v_F(a)=i=i_G(\sigma)>0$ and $v_F(f)$ is nonzero and prime to $p,$ therefore lemma \ref{lem:laurentexpansion} applies and gives that:
    \[
    (1+a)^{v_F(f)}=1+v_F(f)\cdot a+\,\aprod
    \]
    where ``$\aprod$'' denotes elements of greater valuation. Multiplication by $\pi^{v_F(f)}$ gives:
    \[
     \sigma(\pi_F)^{v_F(f)}=\pi_F^{v_F(f)}+v_F(f)\cdot a\cdot \pi_F^{v_F(f)}+\,\aprod
    \]
    where ``$\aprod$'' still denotes greater valuation terms.
    \item[For $\bm{\sigma(\mu)}$:] It is $v_F(\sigma(\mu)-\mu)\geq i+1$ and we write:
    \[
    \sigma(\mu)=\mu+\pi_F^{i+1}\mu',\quad v_F(\mu')\geq 0.
    \]
\end{description}
Putting these together we get:
\begin{align*}
  \sigma(f)&=\sigma(\pi_F)^{v_F(f)}\cdot \sigma(\mu)
  =
  \left(\pi_F^{v_F(f)}+v_F(f)\cdot a\cdot \pi_F^{v_F(f)}+\,\aprod\right)
  \cdot
  (\mu+\pi_F^{i+1}\mu')
  =
  \\
&
=
\pi_F^{v_F(f)}\cdot\mu+v_F(f)\cdot a\cdot \pi_F^{v_F(f)}\cdot \mu+ \pi_F^{v_F(f)+i+1}\mu'
+
v_F(f)\cdot a\cdot \pi_F^{v_F(f)+i+1} \mu'+\,\aprod.
\end{align*}
Then $\sigma(f)-f$ is:
\[
  \sigma(f)-f=v_F(f)\cdot a\cdot \pi_F^{v_F(f)} \cdot\mu+ \pi_F^{v_F(f)+i+1}\cdot \mu'+v_F(f)\cdot a\cdot \pi_F^{v_F(f)+i+1} \cdot\mu'+\,\aprod,
\]
where ``$\aprod$'' still denotes terms with greater valuation. We compare the valuations of the first three terms.

 \begin{itemize}
   
   \item  The valuation of $v_F(f)\cdot a\cdot \pi_F^{v_F(f)}\cdot \mu$ is:
  \begin{equation}\label{equation:1}
    0+i+v_F(f)+v_F(\mu)
    =
    i+v_F(f), 
  \end{equation}
  where the valuation of $v_F(f)$ is zero regardless if $\Char F$ is zero or positive (for $\Char F=0,$ recall that $v_F(f)$ is a prime to $p$ integer)  and also $v_F(\mu)
  =0.$
  \item The valuation of $\pi_F^{v_F(f)+i+1}\cdot \mu'$ is:
  \begin{equation}\label{equation:2}
    v_F(f)+i+1+v_F(\mu')\geq v_F(f)+i+1.
  \end{equation}
  \item The valuation of $v_F(f)\cdot a\cdot \pi_F^{v_F(f)+i+1}\cdot \mu'$ is:
  \begin{equation}\label{equation:3}
    0+i+v_F(f)+i+1+v_F(\mu').
  \end{equation}
 \end{itemize}
 Since $\mu'$ has nonnegative valuation, the (strictly) minimal valuation is that of equation \ref{equation:1}: $i+v_F(f)=i_G(\sigma)+v_F(f).$\end{proof}
\subsection{Ramification jumps}

As before consider a finite Galois extension $F/K$ of complete discrete valuation fields. The inertia subgroup of $G,$ i.e. $G_0,$ is a semidirect product of a cyclic prime to $p$ subgroup by a $p-$subgroup. If the characteristic $p$ of the residue field $\bar{K}$ of $K$ is zero then $G_1=\{1\}.$ We always assume $p$ to be positive. 
 We also assume that $F/K$ is totally ramified, so that $\Gal(F/K)=G_0.$
 If $G_j>G_{j+1}$ we say that $j$ is a \emph{jump} of the ramification filtration.
Since $p$ is positive, the group $G_1$ admits a filtration by elementary abelian extensions:
\[
  G_0\geq G_1= G_{b_1}\gneq G_{b_2}\gneq\dots\gneq G_{b_{h}}\gneq G_{b_{h+1}}=\{1\},
\] 
and $G_{b_i}/G_{b_{i+1}}\simeq (\mathbb{Z}/p\mathbb{Z})^{n_i},$ where $n_i\geq 1.$

The quotient $G/G_1$ amounts for the existence of tame ramification in the extension $F/K.$ Denoting the fixed field $F^{G_1}$ of $F$ by $G_1$ with $F_1$ (so that $K<F_1<F$) it is known 
that for any choice of uniformizers $\pi_K,$ $\pi_{F_1}$ of $K$ and $F_1$ respectively, there are units $\nu\in A_{F_1}$ and $\lambda\in A_K$ so that $(\pi_{F_1}\cdot \nu)^{[F_1:K]}=\pi_K\cdot \lambda.$

In what follows, \textbf{we deal exclusively with non tame (i.e. wild) ramification.} This is done by replacing the extension $F/K$ with $F/F_1=F/F^{G_1}.$ Since a ramification jump at zero does not affect the upper ramification filtration, this approach is sensible for our study.

 After discarding tame ramification,  we deal with the subextension $F/F_1$ which corresponds to the normal subgroup $G_1=G_{b_1}=\Gal(F/F_1)$ of $\Gal(F/K).$ The former is the $p-$part of the latter. We write $G$ for $\Gal(F/F_1).$ 
Set $F_i=F^{G_{b_i}}.$ Since every ramification group is normal in $G$ we have a sequence of Galois extensions:
\[
  K\leq F_{1}\lneq \dots\lneq F_{h}\lneq F_{h+1}=F.
\]
As before, denote by $A_{F_i}$ (resp. $\pi_{F_i}$) the valuation ring (resp. the uniformizer) of $F_i.$

Regarding the divisibility of the ramification jumps, we need to distinguish between two cases: 
\begin{itemize}
\item No jump is divisible by $p=\Char \bar{K},$ or
\item $p$ divides all the jumps (see \cite[IV.2. prop. 11]{serre2013local}). In this case, the characteristic of $F$ is zero, the extension $F/F_1$ is cyclic of order $p^h$ (where $h$ is the number of the jumps) and $p\mid b_1=\frac{b_i}{p^{i-1}}$ for all $i=2,\dots,h$ (see \cite[IV.2 ex.3]{serre2013local}).
\end{itemize}

Let us first deal with the last case, when $p=\Char\bar{K}$ divides the ramification jumps. Write $e_0(F_1)$ and $e_0(F)$ for the absolute ramification indices of $F_1=F^{G_1}$ and $F$ respectively. By  \cite[III.2 prop. 3]{zbMATH0179379}, the ramification filtration of $F/F_1$ is:
\[
  \frac{p\cdot e_0(F_1)}{p-1}=b_1\lneq p\cdot b_1=b_2\lneq\dots\lneq b_h=\frac{e_0(F)}{p-1}.
\]
Furthermore, as the extension $F/F_1$ is given by cyclic extensions of degree $p,$ there are subfields $F_i<F$ corresponding to the subgroups $G_{b_i}$ such that $F_i=F_{i-1}(f_i),$ where $f_i^p=a_i\in F_{i-1}\backslash\{0\}$ and $a_i\neq (1+\pi_{F_{i-1}}u)^p$ for all $u\in A_{F_{i-1}}.$

After dealing with this case, we assume from now on that $p=\Char\bar{K}$ \textbf{does not divide any ramification jump,} i.e. $(p,b_i)=1$ for all $i=1,\dots,h.$
Let us collect here the numerical relations between the various subgroups and subfields. For $i=1,\dots, h$ the following hold:
\begin{itemize}
  
  \item $\Gal(F_{i+1},F_{i})=G_{b_i}/G_{b_{i+1}}\simeq (\mathbb{Z}/p\mathbb{Z})^{n_i},\quad\quad$ $[F_{i+1}:F_{i}]=p^{n_i},$

  \item $[F:F_{i}]=p^{n_i+\dots+n_h}$ and $[F_i:K]=p^{n_1+\dots+n_{i-1}}.$
\end{itemize}

\subsubsection*{Structure of intermediate fields}

For each $i=1,\dots,h$ the field $F_i=F^{G_{b_i}}$ is a local field and a finite Galois extension of $F_1.$ The same holds when we consider $F_{i+1}$ over $F_{i}.$ The extension $F_{i+1}/F_i$ is again totally ramified, and considering its ramification filtration, it has a single jump (otherwise $G_{b+1}/G_{b_i}$ would have more than one jumps). We denote the normalized valuation of $F_i$ by $v_i.$ When $v_{h+1}$ appears it means the valuation $v_F$ of $F.$ The following hold:
\begin{itemize}
\item  $v_F=p^{n_{i}+\dots+n_h}\cdot v_i $ and $v_i=p^{n_1+\dots+n_{i-1}}\cdot v_{F_1}.$ 
\item The extension $F_i/F_1$ corresponds to the quotient $G/G_{b_i}$ and has ramification filtration:
\[
  \left(\frac{G}{G_{b_i}}\right)=
  \left(\frac{G_{b_1}}{G_{b_i}}\right)=
  \left(\frac{G}{G_{b_i}}\right)_{b_1}\gneq
  \left(\frac{G}{G_{b_i}}\right)_{b_2}=
  \left(\frac{G_{b_2}}{G_{b_i}}\right)
  \gneq\dots\gneq
  \left(\frac{G}{G_{b_i}}\right)_{b_i}=
  \left(\frac{G_{b_i}}{G_{b_i}}\right)
  =\{1\}.
\]
Its lower ramification jumps are $b_1,\dots, b_{i-1}.$
\item 
The group $\Gal(F_i/F_{i-1})$ has a unique jump and in particular:
\[
  \Gal(F_i/F_{i-1})\simeq \frac{G_{b_{i-1}}}{G_{b_i}}
  =
  \left(\frac{G_{b_{i-1}}}{G_{b_i}}\right)_0=
  \left(\frac{G_{b_{i-1}}}{G_{b_i}}\right)_{b_{i-1}}\gneq \left(\frac{G_{b_{i-1}}}{G_{b_i}}\right)_{b_{i-1}+1}=\{1\}.
\]
\end{itemize}
\begin{minipage}[l]{0.6\textwidth}
These are explained in great detail in \cite[III.3.6]{zbMATH0179379} and also in \cite[pg. 62-64]{serre2013local}.

 For each $i=1,\dots,h,$ let $f_i\in F_{i+1}$ such that 
 $$\boxed{v_{i+1}(f_i)=-b_i.}$$


\begin{remark}
  We are going to use lemma \ref{thm:mainlemma}. Since $f_{i}$ is in $F_{i+1},$ we will apply the theorem there. Recall that all jumps are assumed to be \emph{prime to $p$.}

      Let us review the situation. $F_{i+1}$ is an extension of both $F_1$ and $F_{i}$ with Galois groups:
  \begin{align}
   \tag{$\star$}\label{[Fi+1/F_1]} \Gal(F_{i+1}/F_1)&
   \simeq 
   \frac{\Gal(F/K)}{\Gal(F/F_{i+1})}
   \simeq
    \frac{G}{G_{b_{i+1}}},
    \\
    \notag\quad \Gal(F_{i+1}/F_{i})&\simeq \frac{\Gal(F/F_{i})}{\Gal(F/F_{i+1})}\simeq \frac{G_{b_{i}}}{G_{b_{i+1}}}.
  \end{align}
  \end{remark}
    When working in the extension $F_{i+i}/F_{i}$ the order function $i_{G_{b_{i}}/G_{b_{i+1}}}$ satisfies, for $\sigma\in G_{b_i},$ the following:
\end{minipage}
\noindent
\begin{minipage}[r]{0.4\textwidth}
\[
\xymatrix{
  F=F_{h+1}=F_{h}(f_h)\ar@{-}[d]^{p^{n_h}} 
\ar@{-}@/_6pc/_{|G|}[dddd]
  \\
  F_{h}=F_{h-1}(f_{h-1}) \ar@{-}[d]^{p^{n_{h-1}}} \\
  F_{h-1}  \ar@{.}[d] \\
F_2=F_1(f_1) \ar@{-}[d]^{p^{n_1}}\\
  F_1
   \ar@{-}[d] \\
  K
}
\]
\end{minipage}

 \[
  i_{\frac{G}{G_{b_j}}}(\sigma G_{b_j})\,=\,
  \begin{cases}
      \infty, & \text{if $j\leq i$}\\
            i_G(\sigma), & \text{if $j>i$}
     \end{cases}
     \quad\text{ and }\quad
      i_{\frac{G_{b_i}}{G_{b_{i+1}}}}(\sigma|_{F_{i+1}})\,=\,
  \begin{cases}
      \infty, & \text{if $i_G(\sigma)> b_i$}\\
            b_i, & \text{if $i_G\left(\sigma\right)=b_i$}
     \end{cases}
  \]
See \cite[IV.1 prop. 3 \& the corollary after]{serre2013local}.

\begin{proposition}\label{prop:fiisnotinFi-1}
  For each $i=1,\dots,h$ the element $f_{i}\in F_{i+1}$ is \emph{not} in $F_{i}.$
\end{proposition}
\begin{proof}
Let $\sigma\in G_{b_{i}}\backslash G_{b_{i+1}}$ so $i_G(\sigma)=b_{i}.$ Restricting  $\sigma$ to $F_{i+1}$ and considering the previous remark, we see that equality of lemma \ref{thm:mainlemma} applies and we have:
\[
  v_{i+1}[\sigma(f_{i})-f_{i}]=-b_i+i_{\frac{G_{b_{i}}}{G_{b_{i+1}}}}(\sigma|_{F_{i+1}})=-b_i+b_i=0\neq \infty,
\]
therefore $\sigma(f_{i})\neq f_{i}.$
\end{proof}


\subsection{Minimality condition}\label{subsuse:minimalitycondition}
We want to establish a ``unique form'' for the monomials $\kappa f_1^{\lambda_1}\cdots f_h^{\lambda_h}$ for $\kappa\in F_1,$ with respect to their  valuation. This is the analog of the uniqueness of the base-$p$ expansion of rational numbers. The following is a generalization of \cite[lemma 2]{MR4194180}. It also appeared in \cite{Madden78}.


\begin{lemma}\label{lem:differentvaluations}
  Let $k\leq h.$ For $0<i\leq k,$ let $\ell_i,$ $w_i\in\Z$ and suppose that
  for $i=\bm{1},\dots,k$ we have  $ 0\leq \ell_i,w_i\lneq p^{n_i}=[F_{i+1}:F_i]$ and:
\[
  \ell_0p^{\sum_{1}^{k}n_i}+\sum_{i=1}^{k-1}\ell_ip^{n_{\bm{i+1}}+\dots+n_k}b_i+\ell_kb_k
  =
  w_0p^{\sum_{1}^{k}n_i}+\sum_{i=1}^{k-1}w_ip^{n_{\bm{i+1}}+\dots+n_k}b_i+w_kb_k,
\]
then $\quad (\ell_{0},\dots,\ell_k)=(w_{0},\dots,w_k).$
\end{lemma}
\begin{proof}
  
 Recall that all $b_i$ are prime to $p$. We start from the end, i.e. for $i=k:$ 
\[
  p^{n_{k}}\left(\ell_0p^{\sum_{1}^{{k-1}}n_i}+\sum_{i=1}^{{k-1}}\ell_ip^{n_{\bm{i+1}}+\dots+n_{k-1}}b_i
  -
  w_0p^{\sum_{1}^{{k-1}}n_i}-\sum_{i=1}^{{k-1}}w_ip^{k-i}b_i\right)
  =
  (w_k -\ell_k)b_k.
\]
Since $0\leq \ell_k,w_k<p^{n_k},$ it must be $\ell_k=w_k.$ The same argument applies after cancelling out the powers of $p$ and repeating.

\end{proof}

\begin{lemma}
Let $\ell_0,w_0\in F_1$ and for $i\in \{1,\dots,h\},$ let $\ell_0f_1^{\ell_1}\cdots f_i^{\ell_i}$ and  $w_0f_1^{w_1}\cdots f_i^{w_i}\in F_{i+1}$ such that:
   \begin{itemize}
   \item $v_{i+1}(\ell_0f_1^{\ell_1}\cdots f_i^{\ell_i})=v_{i+1}(w_0f_1^{w_1}\cdots f_i^{w_i})$ and
   \item $0\leq \ell_j,w_j\lneq p^{n_{j}}=[F_{j+1}:F_j]$ for all $j\in\{1,\dots,i\}.$
\end{itemize}

    Then $v_{\bm{F_1}}(\ell_0)=v_{\bm{F_1}}(w_0)$ and $\ell_j=w_j$ for all $j=1,\dots i.$ 
\end{lemma}
\begin{proof}Recall that $v_{i+1}(f_i)=-b_i.$ Write:\begin{align*}
    v_{i+1}(\ell_0f_1^{\ell_1}\cdots f_i^{\ell_i})&\,
    =
    \,v_{i+1}(w_0f_1^{w_1}\cdots f_i^{w_i})
    \Leftrightarrow \\
  p^{n_1+\dots+n_{i}}v_{\bm{F_1}}(\ell)-\ell_1p^{n_2+\dots+n_{i}}b_1-\dots-\ell_ib_i
    &=
    p^{n_1+\dots+n_{i}}v_{\bm{F_1}}(b)-w_1p^{n_2+\dots+n_{i}}b_1-\dots-w_ib_i,
  \end{align*}
  and apply lemma \ref{lem:differentvaluations}.
\end{proof}

\begin{corollary}\label{cor:differentvaluations}
  If $\ell_0f_1^{\ell_1}\cdots f_i^{\ell_i},$  $w_0f_1^{w_1}\cdots f_i^{w_i}\in F_{i+1},$ $i>0,$ satisfying the requirements of the previous lemma and $\ell_j\neq w_j$ for any $j=\bm{1},\dots,i$ then \emph{their valuations in $F_{i+1}$ differ.}
\end{corollary}

Set $\min_{f_i,F_{i}}(X)\in F_{i}[X]$ for the minimal polynomial of $f_i$ over $F_{i}.$ By prop. \ref{prop:fiisnotinFi-1}, $\deg \min_{f_i,F_{i}}(X)>0$ and 
since $F_{i+1}/F_{i}$ is elementary abelian, the degree of the minimal polynomial of $f_i$ is a power of $p.$

\begin{remark}
Last remark also holds in general: If $\ell_j\neq w_j$ for any $j=0,\dots,i$ then the valuations differ.
\end{remark}

\begin{theorem}\label{thm:defminfi}
  For all $i=1,\dots,h:\quad \deg\min_{f_i,F_{i}}(X)=p^{n_i}=[F_{i+1}:F_i].$
\end{theorem}\begin{proof}
  We have established that the degree of the minimal polynomial is a power of $p,$ smaller than or equal to the degree of the corresponding elementary abelian extension. Set $D_i:=-\min_{f_i,F_{i}}(0)\in F_i.$ We use induction:
  \begin{description}
    
    \item[$\bm{i=1}$:] $\min_{f_1,F_{1}}(f_1)=0$ so that $\sum_{j=1}^{p^a}\ell_if_1^{j}=D_1\in F_1.$ The term on the left side of the last equation features only positive powers of $f_1,$ therefore, by corollary \ref{cor:differentvaluations}, it has a minimal term with respect to the valuation $v_2$ of $F_2$, say $\ell_jf_1^j.$ Evaluating the last equation yields: 
    \begin{equation*}
      v_2(\ell_j)+jv_2(f_1)=v_2(D_1)\Leftrightarrow p^{n_1}v_{\bm{F_1}}(\ell_j)-jb_1=p^{n_1}v_{\bm{F_1}}(D_1).
    \end{equation*}
    Since $0<j\leq p^{n_1}$ and no term has infinite valuation, lemma \ref{lem:differentvaluations} yields that $j=p^{n_1}.$

    \item[$\bm{i}$:] Suppose it holds for $i-1.$ Then $F_{j}=F_{j-1}(f_{j-1})$ for $2\leq j\leq i.$ This means that any element of $F_{i}$ can be written as a sum of monomials $\ell f_1^{\ell_1}\cdots f_{i-1}^{\ell_{i-1}}$ with $\ell\in F_1$ and $0\leq \ell_j<p^{n_j}$ for all $j=1,\dots,i-1.$ Write $\min_{f_i,F_{i}}(f_i)+D_i=D_i.$ The term  $\min_{f_i,F_{i}}(f_i)+D_i$ features only positive powers of $f_i$ so, by corollary \ref{cor:differentvaluations}, it has a minimal element with respect to the valuation $v_{i+1}$ of $F_{i+1},$ call it $\ell f_1^{\ell_1}\cdots f_{i-1}^{\ell_{i-1}}f_{i}^{\ell_i}.$ Also the constant term $D_i$ is in $F_{i}=F_{i-1}(f_{i-1})$ and we can write it as:
    \[
    D_i
    =
    \kappa+\sum_{(w_1,\dots,w_{i-1})\neq (0,\dots,0)}\kappa'f_1^{w_1}\cdots f_{i-1}^{w_{i-1}}
    =
    \kappa + \mathcal{Pi}\text{ with }\kappa\in F_1,\,\mathcal{P}\in F_i\backslash F_1.
    \]
    By corollary \ref{cor:differentvaluations}, the valuations of $\kappa$ and $\mathcal{Pi}$ 
    in $F_{i+1}$ differ, therefore the valuation $v_{_i+1}(D_i)$ is the minimum of them. If $v_{i+1}(D_i)$ was $v_{i+1}(\kappa)$ then the equality $v_{i+1}(\ell f_1^{\ell_1}\cdots f_{i-1}^{\ell_{i-1}}f_{i}^{\ell_i})=v_{i+1}(\kappa)$ would contradict corollary \ref{cor:differentvaluations} in case $\ell_i<p^{n_i}$ since it is also positive. In case $\ell_i=p^{n_i}$ then $\ell_1=\dots=\ell_{i-1}=0$ and we would have $v_{i+1}(f_i^{p^{n_i}})=v_{i+1}(\kappa)$ which again leads to contradiction since it would give $-p^{n_i}b_i=p^{n_1+\dots+n_{i}}v_{\bm{F_1}}(\kappa)$ with $i>1$ while $(b_i,p)=1.$ Therefore we have:
    \[
    v_{i+1}(\ell f_1^{\ell_1}\cdots f_{i-1}^{\ell_{i-1}}f_{i}^{\ell_i})
    =
     v_{i+1}(\min_{f_i,F_i}(f_i)+D_i)=v_{i+1}(D_i)=v_{i+1}(\kappa'f_1^{w_1}\cdots f_{i-1}^{w_{i-1}}),
    \]
    where, $\ell_i$ in the left side and at least one of the $w_j$ in the right side are nonzero. We have explained that for $j=1,\dots,i-1:$ $0\leq\ell_j,w_j<p^{n_j}$ and $\ell_i$ is positive and at most $p^{n_i}=[F_{i+1}:F_i].$ If $\ell_i$ was less than $p^{n_i},$ corollary \ref{cor:differentvaluations} would lead to a contradiction, therefore $\ell_i$ must be $p^{n_i}.$  \noindent\end{description}\end{proof}

\begin{remark}\label{rmk:slicingargument}
  By the 
  second part of the last proof, we have established the following: For all the $i$'s \emph{except} $i=1,$ the constant term $D_i$ of the minimal polynomial of $f_i$ over $F_i$ has a minimal monomial term which is: 
  $$D_i=\kappa f_{1}^{w_1}\cdots f_{i-1}^{w_{i-1}}+\{\text{greater valuation terms}\},$$
  for some $\kappa\in F_1,$ at least one exponent is nonzero and $0\leq w_j\lneq p^{n_j}$ for all $j=1,\dots,w_{i-1}.$ The minimal polynomial evaluated at $f_i$ gives the equality:
   \begin{equation}\label{eq:generalminimalpolynomial}
     f_i^{p^{n_i}}+\,\aprod\, =\kappa_i f_{1}^{w_1}\cdots f_{i-2}^{w_{i-2}}f_{i-1}^{w_{i-1}}+\,\aprod\,
   \end{equation}
   where both ``$\aprod$'' mean greater valuation terms, $\kappa_i\in F_1\backslash\{0\}$ and $(0,\dots,0)\lneq (w_1,\dots,w_{i-1})\leq (p^{n_1}-1,\dots,p^{n_{i-1}}-1).$
\end{remark}

\section{The Hasse - Arf property}\label{subsec:structuralthm}\label{sec:Hassearf}


\vspace{.3cm}
As before, $F/K$ is a finite Galois extension of complete discrete valuation fields, totally ramified with residue field of positive characteristic $p.$

\subsection{The Hasse - Arf theorem}
 Let $b_1,\dots,b_h$ be the jumps of the ramification filtration in lower numbering. Set $\phi:[-1,\infty)\to [-1,\infty)$ by
\begin{equation}\label{def:phifunction}
  \phi(u)=\int_{0}^u\frac{dt}{[G_0:G_t]}\text{ for }u\geq 0
\end{equation}

and $\phi(u)=u$ for $u\in [-1,0].$ In the particular case where the jumps are $(b_1,\dots,b_h),$ setting $b_0=-1$ and $u=b_i$ we have:
\[
  \phi(b_i)+1=\frac{1}{|G_{b_1}|}\sum_{\lambda=0}^{i-1}(b_{\lambda+1}-b_{\lambda})|G_{b_{\lambda+1}}|.
\]


Define $G^{\phi(i)}:=G_i$. Then the ramification filtration of $\Gal{(F/K)}$ \emph{in upper numbering} is:
\[
  G^{-1}=G\geq G^0=G_0= G^{\phi(b_1)}>\ldots>G^{\phi(b_{h})}>1.
\]
The upper numbering allows for the calculation of the ramification groups of $\Gal{(F/K)}/H$ due to the formula:
\[
  \left(\frac{\Gal{(F/K)}}{H}\right)^i=\frac{G^iH}{H}.
\]
\begin{theorem}[Hasse - Arf]
  If $\,\Gal{(F/K)}$ is abelian then the jumps ramification filtration in upper numbering are integers.
\end{theorem}

In other words, if $\Gal{(F/K)}$ is abelian then $\phi(b_i)\in \Z$ for all $i=1,\ldots,h.$

\subsection{Main theorem}
We assume that the ramification jumps are prime to $p.$



\begin{theorem}\label{thm:hassearfstructure}
    $F/K$ has the Hasse - Arf property if and only if, for all $i>1,$ equation \ref{eq:generalminimalpolynomial} is:
     \begin{equation}\label{eq:Hassearfminimalpolynomial}
     \boxed{\quad
      f_i^{[F_{i+1}:F_i]}+\,\aprod
    =
    \kappa_i  f_{i-1}+\,\aprod
    }
    \end{equation}
for some  $\kappa_i\in K\backslash\{0\}$ and ``$\aprod$'' are greater valuation terms (therefore $[F_{i+1}:F_i]v_{i+1}(f_i)=v_{i+1}(\kappa f_{i-1})$).
\end{theorem}


\begin{proof}[\textbf{Proof of theorem \ref{thm:hassearfstructure}}]
By the definition of $\phi$ in equation \ref{def:phifunction}, we can write for $i=2,\dots,h:$
\begin{align*}
  \phi(b_i)-b_1&=\frac{b_2-b_1}{[G:G_{b_2}]}+\dots+\frac{b_i-b_{i-1}}{[G:G_{b_i}]}=\frac{b_2-b_1}{[F_2:K]}+\dots+\frac{b_i-b_{i-1}}{[F_i:K]}=\\
  &=\frac{b_2-b_1}{[F_2:K]}+\frac{b_3-b_2}{[F_3:K]}+\dots+\frac{b_i-b_{i-1}}{[F_i:K]}.
\end{align*}
For $j=1,\dots, i$ we write $[F_j:K]=p^{n_1+\dots+n_{j-1}}$ for simplicity so the last equation is:
\[
   \phi(b_i)-b_1=\frac{b_2-b_1}{p^{n_1}}+\frac{b_3-b_2}{p^{n_1+n_2}}+\dots+\frac{b_i-b_{i-1}}{p^{n_1+\dots+n_{i-1}}}.
\]
By the general form of a totally ramified extension given in equation \ref{eq:generalminimalpolynomial}:
\[ 
  [F_{i+1}:F_i]v_{i+1}(f_i)=v_{i+1}(\kappa_i f_1^{w_1}\cdots f_{i-1}^{w_{i-1}}).
\]
Which is:
\begin{align*}
    -p^{n_i}b_i=p^{n_1+\dots+n_i}v_K(\kappa_i)-w_1p^{n_2+\dots+n_i}b_1-\dots-w_{i-1}p^{n_i}b_{i-1}
\end{align*}
Cancelling $p^{n_i}$, multiplying by $-1$ and subtracting $b_{i-1}$ from both sides gives:
\begin{equation}\label{eq:jumpsdifference}
    b_i-b_{i-1}=-p^{n_1+\dots+n_{i-1}}v_K(\kappa_i)+w_1p^{n_2+\dots+n_{i-1}}b_1+\dots+(w_{i-1}-1)b_{i-1}.
\end{equation}
Finally recall that $(0,\dots,0)<(w_1,\dots,w_i)\leq(p^{n_1}-1,\dots,p^{n_{i-1}}-1).$

\textbf{Assume $\bm{F/K}$ has the Hasse - Arf property.} Then $p^{n_1+\dots+n_{i-1}}\mid b_{i}-b_{i-1}$ for all $i>2.$ Then it divides the left part of equation \ref{eq:jumpsdifference}, which forces $w_1,\dots,w_{i-2}$ to be zero and $w_{i-1}=1.$

\textbf{Assume $\bm{F/K}$ satisfies condition \ref{eq:Hassearfminimalpolynomial}.} In that case equation \ref{eq:jumpsdifference} becomes $b_i-b_{i-1}=-p^{n_1+\dots+n_{i-1}}v_K(\kappa)$ and applying for all $j<i$ gives the result.
\end{proof}

\subsection{Some examples}

We present some examples of our theorem. One standard way to construct cyclic Galois extension of a field is through Witt theory. In \cite[VI, exercises 46-50]{Lang1993} 
is explained how an equation of the form $Y^n=Y+'X$ (where ``$+'$'' is performed in the Witt ring) corresponds to a Galois abelian extension of $k$ of degree $p^n$ provided that $x_0,$ the first component of $X,$  is not of the form $f^p-f$ for any $f\in k$. Additionally  Obus and Pries in \cite{Obus2010} wrote the formulas for extensions of degree $p^3$ over $k\left((t)\right)$. We will use these formulas to calculate the minimal term for the equation of each sub-extension and show that it verifies our theorem.

\begin{example}
Let $\mathrm{char} k=5$. We consider the extension of $k\left((x)\right)$ of degree $125$ given by the (truncated) Witt vector $(x,0,0)\in W_3(k\left((x)\right))$. The formulas for the extension are given in general form in \cite[example 6.1]{Obus2010}. Applied to our situation (with a slight change in the notation) the field generators are given by:
\begin{align}
  f_1^5-f_1=&x; \quad \quad\quad f_2^5-f_2=\frac{x^5+f_1^5-(x+f_1)^5}{5}:=D_2;\label{equationofy}\\
  f_3^5-f_3=&\frac{x^{25}+f_1^{25}-(x+f_1)^{25}}{25}+\frac{f_2^5-(f_2+\frac{x^5+f_1^5-(x+f_1)^5}{5})^5}{5}:=D_3.\label{equationofw}
\end{align}
This yields a Galois cyclic extension of $k\left((x)\right)$ of degree $125$, given here as a sequence of cyclic Artin-Schreier sub-extensions. Consider the extension of $k\left((x)\right)$ of degree $25.$ It is given by the first two equations, which after some calculations turn out to be:
\[
  f_1^5-f_1=x,\qquad f_2^5-f_2=-x^4f_1-2x^3f_1^2-2x^2f_1^3-xf_1^4.
\]
Denote $F_2=k\left((x)\right)(f_1)$ $F_3=F_2(f_2)$. 
 The constant term of the minimal polynomial of $f_2$ is $D_1:=x$ while $D_2:=-x^4f_1-2x^3f_1^2-2x^2f_1^3-xf_1^4.$ The valuations of the generators at each field 
  are:
\begin{center}
\begin{tabular}{||c | c c c||} 
 \hline
 Field & $v(x)$ & $v(f_1)$ & $v(f_2)$ \\ [0.5ex] 
 \hline\hline
 $k((x))$ & $-1$ &  &  \\ 
 \hline
 $F_1=k\left((x)\right)(f_1)$ & $-5$ & $-1$ &  \\
 \hline
 $F_2=k\left((x)\right)(f_1,f_2)$ & $-25$ & $-5$ & $-21$ \\
 \hline
\end{tabular}
\end{center}
The minimal term of $D_2$ with respect to the valuation of $F_2$ is $-x^4f_1$ (as predicted). The lower jumps of the ramification filtration are $b_1=1$ and $b_2=21$ and $5\mid 21-1.$ 

We now calculate the equation of $f_3.$ 
\begin{multline}
f_3^5-f_3=\label{makrinar}
 x \cdot f_1^4 \cdot f_2^4 - x^6 \cdot f_1^4 \cdot f_2^3 - 2 \cdot x^2 \cdot f_1^4 \cdot f_2^3 + 2 \cdot x^2 \cdot f_1^3 \cdot f_2^4+ 2 \cdot x^{11} \cdot f_1^4 \cdot f_2^2\notag \\
  + x^7 \cdot f_1^4 \cdot f_2^2 + 2 \cdot x^7 \cdot f_1^3 \cdot f_2^3 + 2 \cdot x^3 \cdot f_1^4 \cdot f_2^2 + 2 \cdot x^3 \cdot f_1^2 \cdot f_2^4 -x^{16} \cdot f_1^4 \cdot f_2  - 2 \cdot x^{12} \cdot f_1^4 \cdot f_2 \notag\\
   + 2 \cdot x^{12} \cdot f_1^3 \cdot f_2^2+ 2 \cdot x^8 \cdot f_1^3 \cdot f_2^2 - 2 \cdot x^8 \cdot f_1^2 \cdot f_2^3 - x^4 \cdot f_1^4 \cdot f_2 + x^4 \cdot f_1^3 \cdot f_2^2 + x^4 \cdot f_1^2 \cdot f_2^3 + x^4 \cdot f_1 \cdot f_2^4 \notag\\
     - x^{21} \cdot f_1^4 + x^{17} \cdot f_1^4+ 2 \cdot x^{13} \cdot f_1^4+ x^{13} \cdot f_1^3 \cdot f_2 - x^9 \cdot f_1^3 \cdot f_2 + x^9 \cdot f_1^2 \cdot f_2^2 - x^5 \cdot f_1^3 \cdot f_2 + 2 \cdot x^5 \cdot f_1^2 \cdot f_2^2 - x^5 \cdot f_1 \cdot f_2^3 \notag  \\
     - x \cdot f_1^4 - 2 \cdot x^{22} \cdot f_1^3+ 2 \cdot x^{18} \cdot f_1^3+2 \cdot x^{14} \cdot f_1^3 - x^{14} \cdot f_1^2 \cdot f_2 - x^{10} \cdot f_1^2 \cdot f_2 + x^{10} \cdot f_1 \cdot f_2^2 + 2 \cdot x^6 \cdot f_1^3 + x^6 \cdot f_1^2 \cdot f_2\notag \\
      - x^6 \cdot f_1 \cdot f_2^2 - 2 \cdot x^2 \cdot f_1^3 - 2 \cdot x^{23} \cdot f_1^2- x^{15} \cdot f_1^2 - 2 \cdot x^{11} \cdot f_1^2 + 2 \cdot x^{11} \cdot f_1 \cdot f_2 + x^{11} \cdot f_2^2 - 2 \cdot x^7 \cdot f_1^2 + x^7 \cdot f_2^2 \notag \\
       - 2 \cdot x^3 \cdot f_1^2 - \underline{x^{24} \cdot f_1} - 2 \cdot x^{20} \cdot f_1 + x^{16} \cdot f_1 + 2 \cdot x^{16} \cdot f_2+ x^{12} \cdot f_2- x^8 \cdot f_2 - x^4 \cdot f_1 + x^{17} - x^{13}.
\end{multline}
The calculations we performed were, after expanding all expressions in the right  side of eq. \ref{equationofw}, to reduce them modulo the equation of $f_1$. The equation of $f_2$ was not needed as no power of $f_2$ greater than $4$ appears. Then we reduced the coefficients modulo $5$. But this is \emph{not} the final equation of $f_3$. Indeed one can see that all monomials have degree less than or equal to $25$ and $x$ has the smallest valuation, therefore the minimal term is (the underlined) $x^{24}\cdot f_1.$ But in that case, the characteristic $5$ will divide the valuation of $f_3$ and therefore needs to be substituted. This is by \cite[lem. 3.7.7]{Stichtenothv2009} and also illustrated in \cite[example 5.8.8]{MR2241963}. Indeed we can set: $  f_3:=\bar{f}_3+x^4\cdot f_2.$

Then the extension is the same: $F_3(\bar{f}_3)=F_3(f_3)$ since $x^4\cdot f_2$ is in $F_3.$ 
Substitution gives:
\begin{align}
f_3^5-f_3=&\bar{f}_3^5+x^{20}f_2^5-\bar{f}_3-x^4f_2=\\
=&\bar{f}_3^5-\bar{f}_3+x^{20}(f_2-x^4f_1-2x^3f_1^2-2x^2f_1^3-xf_1^4)-x^4f_2=\\
=&\bar{f}_3^5-\bar{f}_3+x^{20}f_2-x^{24}f_1-2x^{23}f_1^2-2x^{22}f_1^3-x^{21}f_1^4-x^4f_2.
\end{align}

Therefore the equation of $\bar{f}_3$ is:
\begin{multline}
\bar{f}_3^5-\bar{f}_3=-\underline{x^{20}f_2}+\cancel{x^{24}f_1}+\cancel{2x^{23}f_1^2}+\cancel{2x^{22}f_1^3}+\cancel{x^{21}f_1^4}+x^4f_2\label{makrinar2}\\
 +x \cdot f_1^4 \cdot f_2^4 - x^6 \cdot f_1^4 \cdot f_2^3 - 2 \cdot x^2 \cdot f_1^4 \cdot f_2^3 + 2 \cdot x^2 \cdot f_1^3 \cdot f_2^4+ 2 \cdot x^{11} \cdot f_1^4 \cdot f_2^2\notag \\
  + x^7 \cdot f_1^4 \cdot f_2^2 + 2 \cdot x^7 \cdot f_1^3 \cdot f_2^3 + 2 \cdot x^3 \cdot f_1^4 \cdot f_2^2 + 2 \cdot x^3 \cdot f_1^2 \cdot f_2^4 -x^{16} \cdot f_1^4 \cdot f_2  - 2 \cdot x^{12} \cdot f_1^4 \cdot f_2 \notag\\
   + 2 \cdot x^{12} \cdot f_1^3 \cdot f_2^2+ 2 \cdot x^8 \cdot f_1^3 \cdot f_2^2 - 2 \cdot x^8 \cdot f_1^2 \cdot f_2^3 - x^4 \cdot f_1^4 \cdot f_2 + x^4 \cdot f_1^3 \cdot f_2^2 + x^4 \cdot f_1^2 \cdot f_2^3 + x^4 \cdot f_1 \cdot f_2^4 \notag\\
     - \cancel{x^{21} \cdot f_1^4} + x^{17} \cdot f_1^4+ 2 \cdot x^{13} \cdot f_1^4+ x^{13} \cdot f_1^3 \cdot f_2 - x^9 \cdot f_1^3 \cdot f_2 + x^9 \cdot f_1^2 \cdot f_2^2 - x^5 \cdot f_1^3 \cdot f_2 + 2 \cdot x^5 \cdot f_1^2 \cdot f_2^2 - x^5 \cdot f_1 \cdot f_2^3 \notag  \\
     - x \cdot f_1^4 - \cancel{2 \cdot x^{22} \cdot f_1^3}+ 2 \cdot x^{18} \cdot f_1^3+2 \cdot x^{14} \cdot f_1^3 - x^{14} \cdot f_1^2 \cdot f_2 - x^{10} \cdot f_1^2 \cdot f_2 + x^{10} \cdot f_1 \cdot f_2^2 + 2 \cdot x^6 \cdot f_1^3 + x^6 \cdot f_1^2 \cdot f_2\notag \\
      - x^6 \cdot f_1 \cdot f_2^2 - 2 \cdot x^2 \cdot f_1^3 - \cancel{2 \cdot x^{23} \cdot f_1^2}- x^{15} \cdot f_1^2 - 2 \cdot x^{11} \cdot f_1^2 + 2 \cdot x^{11} \cdot f_1 \cdot f_2 + x^{11} \cdot f_2^2 - 2 \cdot x^7 \cdot f_1^2 + x^7 \cdot f_2^2 \notag \\
       - 2 \cdot x^3 \cdot f_1^2 - \cancel{x^{24} \cdot f_1} - 2 \cdot x^{20} \cdot f_1 + x^{16} \cdot f_1 + 2 \cdot x^{16} \cdot f_2+ x^{12} \cdot f_2- x^8 \cdot f_2 - x^4 \cdot f_1 + x^{17} - x^{13}.
\end{multline}
Denote the right side of last equation with $D_3$ and it is easy to verify that $\min D_3=-x^{20}\cdot f_2$ (as expected). Indeed the only monomials with total degree at least $21$ are the following:
\[
  -x^{16}\cdot f_1^4\cdot f_2,\qquad x^{17}\cdot f_1^4,\qquad 2\cdot x^{18}\cdot f_1^3
\]
and one easily sees that $-x^{20}\cdot f_2$ has smaller valuation. Finally for $P$ being the place of $F_4$ above $P_3$, the valuations are as follows:
\[
  v_P(x)=-125,\qquad v_{P}(f_1)=-25 ,\qquad v_{P}(f_1)=-105 ,\qquad v_{P}(\bar{f_3})=-521.        
\]
The lower jumps of the ramification filtration are $b_1=1,b_2 =21 ,\, b_3 =521    $ and $p^2=25\mid 500=b_3-b_2.$

Finally the element $x^4\cdot z$ was chosen to substitute $w$ under the following course of thought: The valuation of $w$ is $-24\cdot 25+5=-605=-121\cdot 5.$ On the other hand any element of the form $w'+x^a\cdot y^b\cdot z^c$ yields the same extension of $F_3$ as does $w$ and the valuation of $x^{a}\cdot y^{b}\cdot z^{c}$ at $P$ is $-a\cdot 125-b\cdot 25-c\cdot 105$. Setting: $ a\cdot 125+b\cdot 25+c\cdot 105=121\cdot 5\Leftrightarrow a\cdot 25+b\cdot 5+c\cdot 21=121,$
we end up with a linear Diophantine equation where only positive solutions are allowed.

A generator of the Galois group can be selected such that it acts on the generators as follows:
\begin{align*}
\sigma(f_1)=f_1+1,\quad \sigma(f_2)=f_2+\frac{f_1^5+1-(f_1+1)^5}{5}\text{ and}
\end{align*}
\[
  \sigma(\bar{f}_3)=\bar{f}_3+\frac{f_1^{25}+1-(f_1+1)^{25}}{25}+\frac{f_2^5-(f_2+\frac{f_1^5+1-(f_1+1)^5}{5})^5}{5}
\]
\end{example}

\begin{example}
The following example comes from the same paper i.e. \cite[example 6.2]{Obus2010}. 
Let $p=2$. The following equations define a Galois cyclic extension of $k\left((x)\right)$ of degree $8$. 
\begin{align}
f_1^2-f_1=x,\quad\quad\quad
f_2^2-f_2=xf_1,\quad\quad\quad
f_3^2-f_3=x^3f_1+f_1^3x+xf_1f_2.\label{eqofw}
\end{align}
This extension is acquired by considering the Witt vector $(x,0,0)\in W_3(k\left((x)\right))$. As in the previous example, the valuation of $f_3$ is a multiple of the characteristic, namely $v(f_3)=14$. This time we define $\bar{f}_3$ by $w:=\bar{f}_3+xf_2$. Then equation (\ref{eqofw}) becomes:
\begin{align}
\bar{f}_3^2-\bar{f}_3&=-x^2f_2^2+xf_2+x^3f_1+f_1^3x+xf_1f_2=\notag\\ 
&=-\cancel{x^3f_1}+x^2f_2+xf_2+\cancel{x^3f_1}+f_1^3x+xf_1f_2.\label{equation:lastwittexample}
\end{align}
The right  side of the last equation is $D_3$ whose minimal term is $x^2f_2.$ 
\begin{center}
\begin{tabular}{||c || c| c| c| c||} 
 \hline
 Field & $v(x)$ & $v(f_1)$ & $v(f_2)$ &$v(\bar{f}_3)$ \\ [0.5ex] 
 \hline\hline
 $k((x))$ & $-1$ &  & & \\ 
 \hline
 $F_2=k\left((x)\right)(f_1)$ & $-2$ & $-1$ & & \\
 \hline
 $F_3=k\left((x)\right)(f_1,f_2)$ & $-4$ & $-2$ & $-3$& \\
 \hline
 $F_4=k\left((x)\right)(f_1,f_2,\bar{f}_3)$ & $-8$ & $-4$ & $-6$&$-11$\\
 \hline
\end{tabular}
\end{center}
The lower jumps of the ramification filtration are $b_1=1,\, b_2=3,\, b_3=11$ and  $ 2\mid b_2-b_1,\qquad 4\mid b_3-b_2.$


\end{example}

\Addresses
\end{document}